\documentclass[english, 12pt]{article}
\usepackage{amsfonts}
\usepackage{amsmath}
\usepackage{amssymb}
\usepackage{amsthm}
\usepackage{amscd}
\usepackage{amscd}
\usepackage{amsthm}

\newtheorem{defi}{Definition}[section]
\newtheorem{theo}{Theorem}[section]
\newtheorem{prop}{Proposition}[section]
\newtheorem{lem}{Lemma}[section]

\newtheorem{rk}{Remark}[section]

\numberwithin{equation}{section}

\def\A{{\cal{A}}}

\def\R{{\mathbb{R}}}
\def\N{\mathbb{N}}

\def\calE{{\cal{E}}}
\def\calF{{\cal{F}}}

\def\calL{{\cal{L}}}

\DeclareMathOperator{\dom}{dom}

\newcommand{\Wa}{W^{\alpha/2,2}(\R^d)}

\newcommand{\Wo}{W_0^{\alpha/2,2}(\Omega)}

\newcommand{\Om}{\Omega}

\newcommand{\alp}{\alpha}

\begin{document}
\bibliographystyle{alpha}

\title{ The heat equation for the  Dirichlet fractional Laplacian with Hardy's potentials:
 properties of minimal solutions and blow-up}

\author{\normalsize   Ali BenAmor\footnote{University of Sousse. Sousse, Tunisia. E-mail: ali.benamor@ipeit.rnu.tn}
}

\date{}
\maketitle
\begin{abstract} Local and global properties of minimal solutions for the heat equation generated by the  Dirichlet fractional Laplacian negatively  perturbed by Hardy's potentials on open subsets of $\R^d$ are analyzed. As a byproduct we obtain instantaneous blow-up of nonnegative solutions in the supercritical case.
\end{abstract}

\section{Introduction}
In this paper, we discuss mainly two questions: 1. Local and global properties of minimal solutions of the heat equation related to the Dirichlet fractional Laplacian on open subsets  negatively perturbed by potentials of the type $\frac{c}{|x|^\alpha},\ c>0$.\\
2. Relying on the results obtained in 1. we shall prove complete instantaneous blow-up of nonnegative solutions for the same  equation  provided $c$ is bigger than some critical value $c^*$.\\
To be more concrete, let $0<\alpha<\min(2,d)$ and $\Om$ be an open subset $\Om\subset\R^d$ containing zero. We designate by  $L_0^\Om:=(-\Delta)^{\frac{\alpha}{2}}|_\Omega$ the fractional Laplacian with zero Dirichlet condition on $\Omega^c$ (as explained in the next section). We consider the associated perturbed heat equation
\begin{eqnarray}
\label{heat1}
\left\{\begin{gathered}
-\frac{\partial u}{\partial t}=L_0^\Om u - \frac{c}{|x|^\alpha}u,
\quad \hbox{in } (0,T)\times\Omega,\\
u(t,\cdot)=0,\ on~~~\Omega^c,\ \forall\,0<t<T\leq\infty\\
u(0,x)= u_{0}(x),~~~{\rm a.e.\ in}\  \Omega,
\end{gathered}
\right.
\end{eqnarray}
where $c>0$ and $u_{0}$ is a nonnegative Borel measurable square integrable  function on $\Om$. The meaning of a solution for the equation (\ref{heat1}) will be explained in the next section.\\
Regarding the first addressed  question, in the paper \cite{benamor-kenzizi}, the authors established existence of nonnegative exponentially bounded solutions on bounded Lipschitz domains provided
\begin{eqnarray}
 0<c\leq c^*:=\frac{2^\alpha\Gamma^2(\frac{d+\alpha}{4})}{ \Gamma^2(\frac{d-\alpha}{4})}.
\end{eqnarray}
They also proved that for $c>c^*$ complete instantaneous blowup takes place, provided $\Om$ is a bounded Lipschitz domain.\\
Concerning properties of solutions,  only partial information are available in the literature. Precisely in \cite[Corollary 5.1]{benamor-JPA} the authors proved that for bounded $C^{1,1}$ domains then under some additional condition  one has  the following asymptotic behavior of nonnegative solutions $u(t,x)$ for large time,
\begin{eqnarray}
u(t,x)\sim c_t |x|^{-\beta(c)}|y|^{-\beta(c)}\delta^{\alpha/2}(x)\delta^{\alpha/2}(y),\ a.e.
\label{asymp0}
\end{eqnarray}
where $0<\beta(c)\leq \frac{d-\alpha}{2}$ and $\delta$ is the distance function to the complement of the domain.\\
In case $\Om=\R^d$, owing to recent results  (see \cite{BogdanHardy}) concerning sharp estimates for the heat kernel of the mentioned evolution equation one can derive precise behavior of  nonnegative solutions of the considered equation. Moreover, in \cite[Corollary 4.11]{BogdanHardy} the authors prove blowup of the heat kernel in the supercritical case on $\R^d$, which implies instantaneous blowup of any nonnegative solution on $\R^d$.\\
However, as long as we know, the second question is still open for general open subsets: It is not clear whether for $c>c^*$ and $\Om$ unbounded any nonnegative solution blows up immediately and completely.\\
In these notes we shall, establish sharp local estimates with respect to the spatial variable, up to the boundary, of a special nonnegative solution (the minimal solution) of the considered heat equation on bounded  sets, in the subcritical case. These estimates will  lead to global sharp $L^p$ regularity property of the solution. We also prove complete instantaneous blowup in the supercritical case for arbitrary domains, regardless boundedness and regularity of the domain. Therefore we solve completely and in a unified manner the question of instantaneous blow-up.\\
Our strategy is described as follows: At first stage we show that in the subcritical case the underlying semigroups have heat kernels. Then we shall establish sharp estimates of the heat kernels near zero of the considered semigroups on bounded sets, which in turns will lead to sharp pointwise estimate of the minimal solution near zero of (\ref{heat1}). The latter results are then exploited to prove the above mentioned properties and to enable us to extend the $L^2$-semigroups to  semigroups on some (weighted) $L^p$-spaces, determining therefore the optimal class of initial data. The main ingredients at this stage are, on one side a transformation procedure by  harmonic functions  that will transform the forms related to the considered semigroups into Dirichlet forms together with  the use of  the celebrated improved Hardy--Sobolev inequality to obtain an upper bound for the heat kernel. On the other side a lower bound for the heat kernel will be established, by using Dynkin--Hunt formula together with the sharp estimates from \cite[Theorem 1.1]{BogdanHardy}.\\
Then the precise description of the pointwise behavior of the heat kernel on bounded sets will deserve among others to establish blowup on open sets.\\
The inspiring point for us were the papers \cite{zuazua,baras-goldstein,cabre-martel} where the problem was addressed and solved for
the Dirichlet Laplacian (i.e. $\alpha=2$). We shall record many resemblances between our results and those found in the latter cited papers though the substantial difference between the Laplacian and the fractional Laplacian.
\section{Preparing results}
From now on we fix an open  subset $\Omega\subset\R^d$ containing zero and a real number $\alpha$ such that $0<\alpha<\min(2,d)$.\\
The Lebesgue spaces $L^2(\R^d,dx)$, resp. $L^2(\Omega,dx)$ will be denoted by $L^2$, resp. $L^2(\Omega)$ and their respective norms will be denoted by $\|\cdot\|_{L^2}$, resp. $\|\cdot\|_{L^2(\Om)}$ .
We shall write $\int\cdots$ as a shorthand for $\int_{\R^d}\cdots$.\\
The letters  $C, C',c_t, \kappa_t$ will denote generic nonnegative finite constants
which may vary in value from line to line.\\
Consider the bilinear symmetric form $\calE$ defined in $L^2$ by
\begin{eqnarray}
\calE(f,g)&=&\frac{1}{2}{\A} (d,\alpha)\int \int \frac{(f(x)-f(y))(g(x)-g(y))}
{|x-y|^{d+\alpha}}\,dxdy,\nonumber\\
D(\calE)&=&W^{\alpha/2,2}(\R^d)
:=\{f\in L^2\colon\,\calE[f]:=\calE(f,f)<\infty\},\,
\label{formula1}
\end{eqnarray}
where
\begin{eqnarray}
{\A}{(d,\alpha)}=\frac{\alpha\Gamma(\frac{d+\alpha}{2})}
{2^{1-\alpha}\pi^{d/2}\Gamma(1-\frac{\alpha}{2})},
\label{analfa}
\end{eqnarray}
is a normalizing constant.\\
Using Fourier transform $\hat f(\xi)=(2\pi)^{-d/2}\int e^{-ix\cdot\xi}f(x)\,dx$, a straightforward computation yields the following identity
(see \cite[Lemma 3.1]{frank})
\begin{eqnarray}
\int |\xi|^\alpha|\hat f(\xi)|^2\,d\xi=\calE[f],\ \forall\,f\in W^{\alpha/2,2}(\R^d).
\label{form-fourier}
\end{eqnarray}
It is well known that $\calE$ is a  Dirichlet form, i.e. it is densely defined bilinear symmetric and closed form  moreover it holds,
\begin{eqnarray}
\forall\,f\in\Wa(\R^d)\Rightarrow f_{0,1}:=(0\vee f )\wedge 1\in\Wa\ {\rm and}\ \calE[f_{0,1}]\leq\calE[f],
\end{eqnarray}
Furthermore $\calE$ is regular, i.e. $C_c(\R^d)\cap \Wa$ is dense in both spaces $C_c(\R^d)$ and $\Wa$. For aspects related to Dirichlet forms we refer the reader to \cite{fukushima-book}.\\
The form  $\calE$   is related (via Kato representation theorem) to the selfadjoint
operator commonly named the fractional Laplacian on  $\R^d$, which we shall denote by  $L_0:=(-\Delta)^{\alpha/2}$. We note that the domain of $L_0$ is the fractional Sobolev space  $W^{\alpha,2}(\R^d)$.
For later purposes we recall the following Hardy's inequality
\begin{eqnarray}
\int \frac{f^2(x)}{|x|^\alpha}\,dx\leq \frac{1}{c^*}\calE[f],\ \forall\,f\in W^{\alpha/2,2}(\R^d),
\label{hardy-global}
\end{eqnarray}
with $1/{c^*}$ being  the best constant in the latter inequality.\\
Henceforth we designate by  $L_0^\Om$, the operator which Dirichlet form in $L^2(\overline\Om,dx)$ is given by
\begin{eqnarray*}
D(\calE_\Om)&=&W_0^{\alpha/2,2}(\Om)\colon=\{f\in W^{\alpha/2,2}(\R^d)\colon\, f=0 ~~~q. e.~on~\Om^c\}\nonumber\\
\calE_\Om(f,g)&=&\calE(f,g)\nonumber\\
&=&\frac{1}{2}\A{(d,\alpha)}\int_\Om\int_\Om \frac{(f(x)-f(y))(g(x)-g(y))}{|x-y|^{d+\alpha}}\,dx\,dy
+\int_\Om f(x)g(x)\kappa_\Om(x)\,dx,
\end{eqnarray*}
where
\begin{eqnarray}
\kappa_\Om(x):=\A(d,\alpha)\int_{\Om^c}\frac{1}{|x-y|^{d+\alpha}}\,dy.
\end{eqnarray}
For every $t\geq 0$ we designate by $e^{-tL_0^\Om}$ the operator semigroup related to $L_0^\Om$. In the case $\Om=\R^d$ we omit the superscript $\Om$ in the notations.\\
It is a known fact (see \cite{bogdan-book}) that $e^{-tL_0^\Om},\ t>0$ has a kernel (the heat kernel) $p_t^{L_0^\Om}(x,y)$ which is symmetric jointly continuous and $p_t^{L_0^\Om}(x,y)>0,\ \forall\,x,y\in\Om$.\\
Let us introduce  the notion of solution for problem (\ref{heat1}).
\begin{defi}
{\rm Let $V\in L^1_{loc}(\Om)$ be nonnegative, $u_0\in L^2(\Om)$ be nonnegative as well and $0<T\leq\infty$.  We say that a Borel measurable function $u:[0,T)\times\R^d\to\R$ is a solution of the heat equation
\begin{eqnarray}
\label{heat2}
\left\{\begin{gathered}
-\frac{\partial u}{\partial t}=L_0^\Om u - Vu,
\quad \hbox{in } (0,T)\times\Omega,\\
u(t,\cdot)=0,\ \ {\rm on}~~~\Omega^c,\ \forall\,0<t<T\\
u(0,\cdot)= u_{0},~~~{\rm on }\,\, \Omega,
\end{gathered}
\right.
\end{eqnarray}
if
\begin{enumerate}
\item $u\in\calL_{loc}^2\big([0,T), L_{loc}^2(\Om)\big)$, where $\calL^2$ is the Lebesgue space of square integrable functions.
\item $u\in L^{1}_{loc}\big((0,T)\times \Om,dt\otimes V(x)\,dx\big)$.
\item For every $0\leq t< T$,  $u(t,\cdot)= 0,\ a.e.$ on $\Om^c$.
\item For every $0\leq t< T$ and every Borel function $\phi:[0,T)\times\R^d$ such that $\mathrm{supp}\,\phi\subset [0,T)\times\Om$, $\phi,\ \frac{\partial \phi}{\partial t}\in L^2((0,T)\times\Om)$,  $\phi(t,\cdot)\in D(L_0)$ and
$$\int_0^t\int_\Om |u(s,x)L_0\phi(s,x)|\,ds\,dx<\infty$$
the following identity holds true
\begin{eqnarray}
\int \big((u\phi)(t,x)-u_0(x)\phi(0,x)\big)\,dx &+&\int_{0}^{t}\int
u(s,x)(-\phi_{s}(s,x)+L_0^\Om\phi(s,x))\,dx\,ds\nonumber\\
&=&\int_{0}^{t}\int u(s,x)\phi(s,x)V(x)\,dx\,ds.
\label{variational}
\end{eqnarray}
\end{enumerate}
}
\end{defi}
For every $c>0$ we denote by $V_c$ the Hardy potential
$$
V_c(x)=\frac{c}{|x|^\alpha},\ x\neq 0.
$$
In \cite{benamor-kenzizi} it is proved that for bounded $\Om$, $V=V_c$ and for $0<c\leq c^*$ equation (\ref{heat1}), with potential $V_c$, has a nonnegative solution, whereas for $c>c^*$ and  $\Om$  a bounded Lipschitz domain then no nonnegative solutions occur. It was recently proved in \cite{BogdanHardy} that the same statements hold true for $\Om=\R^d$. In these notes we shall, among others, fill the gap.\\
In the next section we shall be concerned with properties of a special nonnegative solution which is called {\em minimal solution} or {\em semigroup solution} in the subcritical case, i.e. $0<c<c^*$ and in the critical case, i.e. $c=c^*$.
The connotation minimal solution comes from the following observation (proved in \cite{baras-goldstein} for Dirichlet--Laplacian with Hardy potentials, whereas for Dirichlet fractional Laplacian its proved in \cite{benamor-kenzizi} for bounded domains and in Lemma \ref{domination} for general domains, finally it is proved  in \cite{keller-lenz} in a different context): If $u_k$ is the semigroup solution for the heat equation with potential $V_c\wedge k,\ k\in\N$ and if $u$ is any nonnegative solution of (\ref{heat1}) then $u_\infty:=\lim_{k\to\infty}u_k$ is a nonnegative solution of (\ref{heat1}) and $u_\infty\leq u\ a.e.$.\\
We shall name $u_\infty$ the minimal nonnegative solution and shall denote it by $u$.\\
Let $0<c< c^*$. We denote by $\calE_\Om^{V_c}$ the quadratic form defined by
\begin{eqnarray}
D(\calE_\Om^{V_c} )=\Wo,\ \calE_\Om^{V_c}[f] = \calE_\Om[f] - \int_\Om f^2(x)V_c(x)\,dx.
\end{eqnarray}
Whereas for $c=c^*$, we set
\begin{eqnarray}
\dot{\calE_\Om}^{V_{c^*}}\colon\,  D(\dot{\calE_\Om}^{V_{c^*}} )=\Wo,\ \dot{\calE_\Om}^{V_{c^*}}[f] = \calE_\Om[f] - \int_\Om f^2(x)V_{c^*}(x)\,dx.
\end{eqnarray}
In the case $\Om=\R^d$ we shall omit the subscript $\Om$.\\
As the closability of  $\dot{\calE_\Om}^{V_{c^*}}$ in $L^2(\Om)$ is not obvious we shall perform a method that enables us to prove in a unified manner the closedness of $\calE_\Om^{V_c}$ as well as the closability of  $\dot{\calE_\Om}^{V_{c^*}}$ in $L^2(\Om)$.\\
To that end we recall  some known facts concerning  harmonic functions of $L_0 -\frac{c}{|x|^\alpha}$.\\
We know from \cite[Lemma 2.2]{benamor-JPA} that for every $0<c\leq c^*$ there is a unique $\beta=\beta(c)\in(0,\frac{d-\alpha}{2}]$ such that $w_c(x):=|x|^{-\beta(c)},\ x\neq 0$ solves the equation
\begin{eqnarray}
&&(-\Delta)^{\alpha/2}w-c|x|^{-\alpha}w=0\ {\rm in\ the\ sense\ of\ distributions}.
\label{harmonic1}
\end{eqnarray}
That is
\begin{eqnarray}
< \hat w,|\xi|^{\alpha}\hat\varphi>-c<|x|^{-\alpha}w,\varphi>=0\ \forall\,\varphi\in{\cal S}.
\end{eqnarray}
Making use of Riesz potential it is proved in \cite[Lemma 2.2]{benamor-JPA} that equation (\ref{harmonic1}) is equivalent to
\begin{eqnarray}
\int \frac{w_c(y)}{|x-y|^{d-\alpha}}|y|^{-\alpha}\,dy = c w_c(x),\ \forall\,x\neq 0.
\label{Riesz}
\end{eqnarray}
Furthermore for $\beta_*:=\frac{d-\alpha}{2}$, we have $c=c^*$, i.e., $w_{c^*}(x)=|x|^{-\frac{d-\alpha}{2}},\ x\neq 0$.\\
Next we fix definitively $c\in (0,c^*]$ .\\
For $0<c<c^*$ let $Q_\Om^c$ be the $w_c$-transform of $\calE_\Om^{V_c}$, and for $c=c^*$ let $\dot Q_\Om^{c^*}$ be the $w_{c^*}$-transform of $\dot\calE_\Om^{V_c}$ i.e.,  the quadratic forms defined in $L^2(\Om,w_c^2dx)$ and in $L^2(\Om,w_{c^*}^2dx)$ respectively by:
\begin{eqnarray*}
\dom(Q_\Om^c):=\{f\in L^2(\Om,w_c^2dx)\colon\,w_cf\in\Wo\},\ Q_\Om^c[f]=\calE_\Om^{V_c}[w_cf],\ \forall\,f\in\,\dom(Q_\Om^c),
\end{eqnarray*}
whereas
\begin{eqnarray*}
\dom(\dot Q_\Om^{c^*}):=\{f\in L^2(\Om,w_{c^*}^2dx)\colon\,w_{c^*}f\in\Wo\},\ \dot Q_\Om^{c^*}[f]=\dot\calE_\Om^{V_c}[w_{c^*}f],\ \forall\,f\in\,\dom(\dot Q_\Om^{c^*}).
\end{eqnarray*}
In the case $\Om=\R^d$ we shall omit the subscript $\R^d$ in the above notations.
\begin{lem}
\begin{enumerate}
 \item For every  $0<c<c^*$, the form $Q_\Om^c$ is a Dirichlet form in $L^2(\Om,w_{c}^2dx)$  and
\begin{eqnarray}
Q_\Om^c[f]=\frac{\A(d,\alpha)}{2}\int\int \frac{(f(x)-f(y))^2}{|x-y|^{d+\alpha}} w_c(x)w_c(y)\,dxdy,\ \forall\,f\in \dom(Q_\Om^c).
\label{ID1}
\end{eqnarray}
\item  For $c=c^*$ the form $\dot Q_\Om^{c^*}$ is closable in $L^2(\Om,w_{c}^2dx)$ and
\begin{eqnarray}
\dot Q_\Om^{c^*}[f]=\frac{\A(d,\alpha)}{2}\int\int \frac{(f(x)-f(y))^2}{|x-y|^{d+\alpha}} w_{c_*}(x)w_{c_*}(y)\,dxdy,\ \forall\,f\in \dom(\dot Q_\Om^{c^*}).
\label{ID2}
\end{eqnarray}
Let $Q_\Om^{c^*}$ be the closure of $\dot Q_\Om^{c^*}$. Then $Q_\Om^{c^*}$ is a Dirichlet form. It follows in particular, that $\dot\calE_\Om^{V_{c^*}}$ is closable
\item The sets $C_c^\infty(\Om\setminus\{0\})$ and $\dom (Q_\Om^c)\cap C_c(\Om)$ are cores for  $Q_\Om^c$. It follows that $Q_\Om^c$ is regular for every $0<c\leq c^*$.
\end{enumerate}
\label{closability}
\end{lem}
\begin{rk}
\rm{
We shall show in Remark \ref{NotClosed} that $\dot\calE_\Om^{V_{c^*}}$ is in fact, not closed.
}
\end{rk}
\begin{proof}
The proof of formulae (\ref{ID1})-(\ref{ID2}) follows the lines of the proof of \cite[Lemma 3.1]{benamor-JPA}, where bounded $\Om$ is considered, so we omit it.\\
We turn our attention now to prove the rest of the lemma.\\
Let $0<c<c^*$. Utilizing Hardy's inequality we obtain
\begin{eqnarray}
(1-\frac{c}{c^*})\calE_\Om[f]\leq \calE_\Om^{V_c}\leq\calE_\Om[f],\ \forall\,f\in\Wo,
\end{eqnarray}
from which the closedness of $\calE_\Om^{V_c}$ follows, as well as the closedness of $Q_\Om^c$. From the definition of $Q_\Om^c$ and the fact that   $\calE_\Om^{V_c}$ is densely defined we conclude that $Q_\Om^c$ is densely defined as well. On the other, on the light of formula (\ref{ID1})  it is obvious that the normal contraction acts on $\dom(Q_\Om^c)$ and hence $Q_\Om^c$ is a Dirichlet form.\\
For the critical case formula (\ref{ID2}) indicates that $Q_\Om^{c^*}$ is Markovian and closable, by means of Fubini theorem. Thus, according to \cite[Theorem 3.1.1]{fukushima-book} its closure is a Dirichlet form.\\
To prove claim 3,  we recall that $C_c^\infty(\Om)$ and $\Wo\cap C_c(\Om)$ are cores for $\calE_\Om$ and hence they are cores  for $\calE_\Om^{V_c}, \dot\calE_\Om^{V_c}$, since both forms are dominated by $\calE_\Om$. On the other hand the map $f\mapsto w_c^{-1} f$ maps $C_c^\infty(\Om)$ into $C_c^\infty(\Om\setminus\{0\})$ and $\Wo\cap C_c(\Om)$ into $\dom(Q_\Om^c)\cap C_c(\Om),\ \dom(\dot Q_\Om^{c^*})\cap C_c(\Om)$. All these considerations together with fact that  $\dom(\dot Q_\Om^{c_*})$ is a core for $Q_\Om^{c^*}$ lead to assertion 3.
\end{proof}
Henceforth, we denote by $\calE_\Om^{V_{c^*}}$ the closure of $\dot\calE_\Om^{V_{c^*}}$, by $L_{V_c}^\Om$ the selfadjoint operator associated to $\calE_\Om^{V_{c}}$ for every $0<c\leq c^*$ and by $e^{-tL_{V_c}^\Om},\ t\geq 0$ the related semigroups.\\
Similarly, for every $0<c\leq c^*$ we designate by $A_\Om^{w_c}$ the operator associated to $Q_\Om^c$ in the weighted Lebesgue space $L^2(\Om,w_c^2dx)$ and $T_{t,\Om}^{w_c},\ t\geq 0$ its semigroup. Then
\begin{eqnarray}
 A_\Om^{w_c}=w_c^{-1}L_{V_c}^\Om w_c\ {\rm and}\ T_{t,\Om}^{w_c}=w_c^{-1}e^{-tL_{V_c}^\Om}w_c,\ t\geq 0.
\label{TransformedSg}
\end{eqnarray}
The next proposition explains why are minimal solutions also semigroup solutions.
\begin{prop} For every $0<c\leq c^*$, the minimal solution is given by  $u(t):=e^{-tL_{V_c}^\Om}u_0,\ t>0$. Thus for each $t>0$, $u(t)\in D(L_{V_c}^\Om)$ and  $u\in C([0,\infty);L^2(\Om))\cap C^1((0,\infty);L^2(\Om))$.  Furthermore $u$ fulfills Duhamel's formula
\begin{eqnarray}
u(t,x)&=&e^{-tL_0^\Om}u_0(x)+\int_0^t\int_{\Om} p_{t-s}^{L_0^\Om}(x,y)u(s,y)V_c(y)\,dy\,ds,\ \forall\,t>0,\ a.e. x\in\Om.
\label{Duhamel}
\end{eqnarray}
\label{sg-Sol}
\end{prop}
\begin{proof} Let $(h_k)_k$ be the sequence of closed  quadratic forms in $L^2(\Om)$ defined by
$$
h_k:=\calE_\Om - V_c\wedge k,
$$
and $(H_k)_k$ be the related selfadjoint operators. Then $(h_k)_k$ is uniformly lower semibounded and $h_k\downarrow \calE_\Om^{V_c}$ in the subcritical case, whereas $h_k\downarrow \dot\calE_\Om^{V_{c^*}}$ in the critical case. As both forms $\calE_\Om^{V_c},\ \dot\calE_\Om^{V_{c^*}}$ are closable, we conclude by \cite[Theorem 3.11]{kato} that $(H_k)$ converges in the strong resolvent sense to $L_{V_c}^\Om$ for every $0<c\leq c^*$. Hence $e^{-tH_k}$ converges strongly to $e^{-tL_{V_c}^\Om}$ and then the monotone
sequence $u_k:=e^{-tH_k}u_0$ converges to $e^{-tL_{V_c}^\Om}u_0$ which is nothing else but the minimal solution.\\
The remaining claims of the proposition follow from the standard theory of semigroups.

\end{proof}
As minimal solutions  are given in term of semigroups we are led to analyze properties of the latter objects to gain information about the formers. Here is a first result in this direction.
\begin{prop} For every $t>0$ the semigroup $e^{-tL_\Om^{V_c}},\ t>0$ has a measurable nonnegative symmetric absolutely continuous kernel, $p_{t}^{L_{V_c}^\Om}$, in the sense that for every $v\in L^2(\Om)$ it holds,
\begin{eqnarray}
e^{-tL_\Om^{V_c}}v  =\int_\Om  p_{t}^{L_{V_c}^\Om}(\cdot,y)v(y)\,dy,\ a.e.\ x,y\in\Om,\ \forall\,t>0.
\label{f0}
\end{eqnarray}
\label{ExistKernel}
\end{prop}
We shall call  $p_{t}^{L_{V_c}^\Om}$ the heat kernel of $e^{-tL_\Om^{V_c}}$. Let us emphasize that formula (\ref{f0}) implies that the heat kernel $p_{t}^{L_{V_c}^\Om}$ is finite a.e..
\begin{proof}
Owing to the known facts that $e^{-tL_0^\Om},\ t>0$ has a nonnegative heat kernel and $V_c\wedge k$ is bounded we deduce that $e^{-tH_k}$ has a nonnegative heat kernel as well, which we denote by $P_{t,k}$. Moreover, since the sequence $(V_c\wedge k)_k$ is monotone increasing, we obtain with the help of Duhamel's formula that  the sequence $(P_{t,k})_k$ is monotone increasing as well. Set
\begin{eqnarray}
p_{t}^{L_{V_c}^\Om}(x,y):=\lim_{k\to\infty}P_{t,k}(x,y),\ \forall\,t>0,\ a.e.\ x,y\in\Om.
\end{eqnarray}
Then  $p_{t}^{L_{V_c}^\Om}$ has all the first properties mentioned in the proposition.\\
Let $v\in L^2(\Om)$ be nonnegative. Then by monotone convergence theorem, together with Proposition (\ref{sg-Sol}) we get
\begin{eqnarray}
e^{-tL_\Om^{V_c}}v&=&\lim_{k\to\infty}u_k(t)=\lim_{k\to\infty}e^{-tH_k}v
=\lim_{k\to\infty}\int_\Om  P_{t,k}(\cdot,y)v(y)\,dy\nonumber\\
&=&\int_\Om  p_{t}^{L_{V_c}^\Om}(\cdot,y)v(y)\,dy,\ a.e.\ x,y\in\Om,\ \forall\,t>0.
\end{eqnarray}
For an arbitrary $v\in L^2(\Om)$ formula (\ref{f0}) follows from the last step by decomposing $v$ into its positive and negative parts.
\end{proof}
\begin{rk}
{\rm
From Proposition \ref{ExistKernel} in conjunction with formula (\ref{TransformedSg}), we obtain existence of an absolutely continuous kernel for the semigroups $T_{t,\Om}^{w_c}$ for each $t>0$ and each $0<c\leq c^*$. We shall denote by $q_t^{\Om}$ the already mentioned kernel and we call it the heat kernel of $Q_\Om^c$. For the particular case $\Om=\R^d$ we will omit the superscript $\R^d$. Let us stress that kernels $q_t^{\Om}$ depend upon $c$. However, we shall keep the dependence hidden and emphasize it, whenever it would be relevant.\\
Once again, formula (\ref{TransformedSg}) leads to
\begin{eqnarray}
q_t^\Om(x,y)=\frac{p_t^{L_{V_c}^\Om}(x,y)}{w_c(x)w_c(y)},\ \forall\,t>0,\ a.e.,\ x,y\in\Om.
\label{TransformedKernel}
\end{eqnarray}

}
\label{Transfer}
\end{rk}
In the particular case $\Om=\R^d$, we proceed to establish a very interesting global property for Dirichlet forms $Q^c$, namely conservativeness. The significance of conservativeness lies among others, in the fact that the Hunt process associated with $Q^c$ can start at quasi-every  point and has an infinite life time.

To achieve our goal we introduce the following forms.
We fix $0<c\leq c_*$ and define the forms $\dot\calF^c$ by:
\begin{align}
\dom(\dot\calF^c)&= C_c^\infty(\R^d),\
 \dot\calF^c[f]=  \frac{\A(d,\alpha)}{2}\int\int \frac{(f(x)-f(y))^2}{|x-y|^{d+\alpha}} w_c(x)w_c(y)\,dxdy,\ \forall\,f\in C_c^\infty(\R^d).
\end{align}
\begin{lem}
\begin{enumerate}
\item The quadratic form $Q^c$ is well defined, in the sense that $\dot\calF^c[f]<\infty$ for every $f\in C_c^\infty(\R^d)$. Moreover it is closable in $L^2(\R^d,w_c^2dx)$.\\
Let
$$
\calF^c= \text{the closure of } \dot\calF^c \text{ in } L^2(\R^d,w_c^2dx).
$$
\item  For $c<c^*$, it holds $Q^c=\calF^c$.
\end{enumerate}
\end{lem}
\label{Equality}
\begin{proof}

The first part of assertion 1., is indeed equivalent to the following two conditions (see \cite[Example 1.2.1]{fukushima-book}): for every compact set $K$ and every open set $\Om_1$ with $K\subset\Om_1$ one should have
\begin{eqnarray*}
\int_{K\times K}|x-y|^{2-d-\alpha}w_c(x)w_c(y)\,dx\,dy<\infty,\
\int_{K}\int_{\Om_1^c}|x-y|^{-d-\alpha}w_c(x)w_c(y)\,dx\,dy<\infty.
\end{eqnarray*}
The first part of the latter conditions was already proved for bounded sets in \cite[Lemma 3.1]{benamor-JPA}. Let us prove the finiteness of the second integral.\\
{\em Case 1}: $0\in K$. Then $0\not\in\Om_1^c$ and $\sup_{y\in\Om_1^c}w_{c}(y)<\infty$. Set $\delta:=dist(K,\Om_1^c)>0$. Then $\delta>0$. Thus for every $x\in K$ we get
\begin{eqnarray}
\int_{\Om_1^c}|x-y|^{-d-\alpha}\,dy\leq \int_{\{|x-y|>\delta\}}|x-y|^{-d-\alpha}\,dy\leq C<\infty.
\end{eqnarray}
Hence the second integral is finite.\\
{\em Case 2}: $0\not\in K$. Then $\sup_{x\in K}w_{c}(x)<\infty$. Let  $x\in K$. Making use of identity (\ref{Riesz}) we obtain
\begin{align}
\int_{\Om_1^c\cap B_1} |x-y|^{-d-\alpha}w_c(y)\,dy\leq \delta^{-2\alpha}\int_{\Om_1^c\cap B_1} \frac{w_c(y)}{|x-y|^{d-\alpha}}|y|^{-\alpha}\leq C w_c(x),
\end{align}
and
\begin{align}
\int_{\Om_1^c\cap B_1^c} |x-y|^{-d-\alpha}w_c(y)\,dy\leq \int_{\{|x-y|>\delta\}}|x-y|^{-d-\alpha}\,dy\leq C<\infty.
\end{align}
Hence, once again the second integral is finite and the first part of assertion 1. is proved.\\
The proof of closability is a standard matter so we omit it.\\
2.: Let $0<c<c^*$. First we show  $C_c^\infty(\R^d)\subset \dom(Q^c)$. Let $f\in C_c^\infty(\R^d)$. We have to prove $w_c f\in\Wa$. On the one hand $w_c f\in L^2(\R^d)$. On the other hand, following the lines of the proof of \cite[Lemma 3.1]{benamor-JPA} we obtain
\begin{align}
\int\int \frac{(w_c(x)f(x) - w_c(y)f(y))^2}{|x-y|^{d+\alpha}}\,dx\,dy&= \int\int \frac{(f(x) - f(y))^2}{|x-y|^{d+\alpha}}w_c(x)w_c(y)\,dx\,dy\nonumber\\
& + \int f^2(x) w_c^2(x) V_c(x)\,dx.
\end{align}
We already proved that the first integral is finite. Since $c<c^*$ and $f\in C_c^\infty(\R^d)$, then
$$
\int f^2(x) w_c^2(x) V_c(x)\,dx<\infty.
$$
Hence $w_c f\in\Wa$ and $f\in D(Q^c)$.\\
Let us recall, by Lemma \ref{closability}-3, $C_c^\infty(\R^d\setminus\{0\})$ is a form core for $Q^c$. Regarding the inclusion $C_c^\infty(\R^d\setminus\{0\})\subset C_c^\infty(\R^d)$, the latter space is also a core for $Q^c$. Furthermore,  $C_c^\infty(\R^d)$ is obviously a core for $\calF^c$.  Hence forms $Q^c$ and $\calF^c$ coincide the common core $C_c^\infty(\R^d)$. Thereby they are identical and the proof is completed.
\end{proof}
%
\begin{theo}
Assume that $\Om=\R^d$. Then for every $0<c\leq c^*$ the form $Q^c$ is conservative. It follows, in particular,
\begin{eqnarray}
\int_{\R^d} p_{t}^{L_{V_c}^{\R^d}}(x,y) w_c(y)\,dy=w_c(x),\ \forall\,x\neq 0.
\label{Totalmass}
\end{eqnarray}
\label{conservative}
\end{theo}
\begin{proof}
Identity (\ref{Totalmass}) is an immediate consequences of the conservativeness property which we proceed to prove.\\
As a  first step we shall prove conservativeness in the subcritical case.\\
{\em The subcritical case.} Let $0<c<c^*$. On the light of Lemma \ref{Equality}-2, we shall use Masamune's result \cite{Masamune}, which asserts in our special case: If
\begin{eqnarray}
\sup_x w_c^{-1}(x)\int_{\R^d} (1\wedge |x-y|^2)|x-y|^{-d-\alpha}w_c(y)\,dy<\infty,
\label{UniformBound}
\end{eqnarray}
and for some $a>0$,
\begin{eqnarray}
\int_{\R^d} e^{-a|x|} w_c^2(x)\,dx<\infty,
\label{Integrability}
\end{eqnarray}
then the form $Q^c$ is conservative.\\
Clearly condition (\ref{Integrability}) is fulfilled.\\
Let us show that condition (\ref{UniformBound}) is satisfied as well. We recall $w_c(x)=|y|^{-\beta(c)}$ for some $\beta:=\beta(c)\in(0,\frac{d-\alpha}{2})$. Let
\begin{eqnarray}
 I_1(x):=\int_{B_1(x)} \frac{|y|^{-\beta}}{|x-y|^{d+\alpha-2}}\,dy,\ \alpha'=2-\alpha.
\end{eqnarray}
Let $|x|\leq 2$ and  $\gamma:=\frac{d-\alpha}{2}$. Then
\begin{eqnarray}
 I_1(x)&=&\int_{B_1(x)} \frac{|y|^{-\beta} |y|^{\alpha'}}{|x-y|^{d-\alpha'}}|y|^{-\alpha'}\,dy\nonumber\\
 &\leq& 2^{\alpha'} \int_{B_1(x)} \frac{|y|^{-\beta}}{|x-y|^{d-\alpha'}} |y|^{-\alpha'}\,dy.
\end{eqnarray}
In the case $\alpha\geq 1$, we obtain
$$
\alpha'>0,\ 0<\beta<d-\alpha'.
$$
Thus we apply identity  (\ref{Riesz}) to get
\begin{eqnarray}
\int_{B_1(x)} \frac{|y|^{-\beta}}{|x-y|^{d-\alpha'}}|y|^{-\alpha'}\,dy&\leq& \int_{\R^d}
 \frac{|y|^{-\beta}}{|x-y|^{d-\alpha'}}|y|^{-\alpha'}\,dy\nonumber\\
&=&C w_c(x),
\end{eqnarray}
and
$$
I_1(x)\leq C  w_c(x).
$$
In the case $0<\alpha<1$, change $2-\alpha$ by $\alpha_1=\frac{1-\alpha}{2}$ to obtain (by similar arguments)
\begin{eqnarray}
I_1(x)\leq C\int_{B_1(x)} \frac{|y|^{-\beta}}{|x-y|^{d-\alpha_1}}|y|^{-\alpha_1}\,dy
\leq C w_c(x).
\end{eqnarray}
Let now $|x|\geq 2$. Then for every $y\in B_1(x)$ we have $|y|\geq |x|-1\geq1$. Thus
$$
I_1(x)\leq C  \frac{1}{(|x|-1)^{\beta}}.
$$
Hence in both cases we obtain
$$
\sup_x  w_c(x) I_1(x)<\infty.
$$
For the remaining integral, let
\begin{eqnarray}
 I_2(x):=\int_{B_1^c(x)} \frac{|y|^{-\beta}}{|x-y|^{d+\alpha}}\,dy.
\end{eqnarray}
We decompose the integral into the sum of three integrals
\begin{eqnarray}
I_2(x)&=&\int_{B_1^c(x)\cap\{|y|< 1\}} \frac{|y|^{-\beta}}{|x-y|^{d+\alpha}}\,dy
+\int_{B_1^c(x)\cap\{|y|> 1\wedge |x|/2\}} \frac{|y|^{-\beta}}{|x-y|^{d+\alpha}}\,dy\nonumber\\
&+&\int_{B_1^c(x)\cap\{1<|y|<|x|/2\}} \frac{|y|^{-\beta}}{|x-y|^{d+\alpha}}\,dy.
\end{eqnarray}
On the set $B_1^c(x)\cap\{|y|< 1\}$ we have $|x-y|^{-d-\alpha}\leq |x-y|^{-d+\alpha}|y|^{-\alpha}$. Thus
\begin{eqnarray}
\int_{B_1^c(x)\cap\{|y|< 1\}} \frac{|y|^{-\beta}}{|x-y|^{d+\alpha}}\,dy
&\leq& \int_{B_1^c(x)\cap\{|y|< 1\}} \frac{|y|^{-\beta}}{|x-y|^{d-\alpha}}|y|^{-\alpha}\,dy\nonumber\\
&\leq& \int_{B_1^c(x)}\frac{|y|^{-\beta}}{|x-y|^{d-\alpha}}|y|^{-\alpha}\,dy
\leq C|x|^\beta.
\end{eqnarray}
Furthermore
\begin{eqnarray}
\int_{B_1^c(x)\cap\{|y|> 1\wedge |x|/2\}} \frac{|y|^{-\beta}}{|x-y|^{d+\alpha}}\,dy
\leq  2^{\beta} |x|^{-\beta}\int_{B_1^c(x)} |x-y|^{-d-\alpha}\,dy
\leq C|x|^{-\beta}.
\end{eqnarray}
For the last integral we have two situations: if the set $E:=B_1^c(x)\cap\{1<|y|<|x|/2\}$ is empty, then we are done. If not, then on the set $E$, it holds
\begin{eqnarray}
|x-y|\geq \frac{|x|}{2}\geq |y|>1.
\end{eqnarray}
Hence
\begin{eqnarray}
\int_E \frac{|y|^{-\beta}}{|x-y|^{d+\alpha}}\,dy&\leq&
2^{\beta} |x|^{-\beta}\int_E \frac{|y|^{-\beta}}{|x-y|^{\frac{d+3\alpha}{2}}}\,dy\nonumber\\
&\leq& 2^{\beta} |x|^{-\beta}\int_E \frac{|y|^{-\beta}}{|y|^{\frac{d+3\alpha}{2}}}\,dy
\leq C  w_c(x).
\end{eqnarray}
Finally we get $\sup_{x} w_c(x) I_2(x)<\infty$.\\
Putting all together we get that condition (\ref{UniformBound}) is fulfilled. Thereby the form $Q^c$ is conservative.

{\em The critical case:} We recall that conservativeness means
$$
T_t^{w_{c^*}}1 =1.
$$
Here $T_t^{w_{c^*}}$ is the $L^\infty$-semigroup related to $Q^{c^*}$. Let $(\varphi_k)\subset C_c(\R^d)$ be a sequence of positive functions such that $\varphi_k\uparrow 1$. Then from the standard construction of the $L^\infty$-semigroup ($(\varphi_k)\subset L^2(\R^d,w_{c^*}^2dx)\cap L^\infty(\R^d)$) together with Remark \ref{Transfer}, we achieve
\begin{align*}
T_t^{w_{c^*}}1=\lim_{k\to\infty}\int q_t (\cdot,y)\varphi_k(y) w_{c^*}^2(y)\,dy.
\end{align*}
An application of monotone convergence theorem leads to
$$
T_t^{w_{c^*}}1= \int q_t (\cdot,y) w_{c^*}^2(y)\,dy.
$$
From the contraction property of the $L^\infty$-semigroup related to $Q^{c^*}$ we derive
$$
\int q_t (\cdot,y) w_{c^*}^2(y)\,dy\leq 1,
$$
which leads to
\begin{eqnarray}
\int_{\R^d} p_{t}^{L_{V_{c^*}}^{\R^d}}(x,y) w_{c^*}(y)\,dy\leq w_{c^*}(x),\ \forall\,x\neq 0,\ t>0.
\label{Ineq1}
\end{eqnarray}
Now the first part of the proof yields, for every $0<c<c^*$,
\begin{align}
w_c(x)=\int_{\R^d} p_{t}^{L_{V_c}^{\R^d}}(x,y) w_c(y)\,dy&\leq \int_{\R^d} p_{t}^{L_{V_{c^*}}^{\R^d}}(x,y) w_{c}(y)\,dy\nonumber\\
&=\int_{B_1} p_{t}^{L_{V_{c^*}}^{\R^d}}(x,y) w_{c}(y)\,dy + \int_{B_1^c} p_{t}^{L_{V_{c^*}}^{\R^d}}(x,y) w_{c}(y)\,dy.
\end{align}
Let us observe that the first integrant is increasing, whereas the second one is decreasing with respect to $c$.  Hence, letting $c\to c^*$ and combining monotone convergence theorem  with inequality (\ref{Ineq1}) we achieve $T_t^{w_{c^*}}1 =1$ and the proof is completed.
\end{proof}
\begin{rk}
{\rm
Theorem \ref{conservative} was proved in \cite[Theorem 3.1]{BogdanHardy} and \cite[Theorem 2.4]{Jakubowski2019} however with a different method using integral analysis. Our proof is different.

}
\end{rk}
\section{Heat kernel estimates, local and global behavior of the minimal solution in space variable}
Along this section we assume that $\Om$ is bounded.\\
Since potentials $V_c$ are too singular (they are not in the Kato-class, for example), investigations of properties of solutions of the  evolution equations related to $L_0^\Om - V_c$ becomes a delicate problem. In fact, the theory of elliptic regularity is no more applicable in this context. To overcome the difficulties we shall make use of  the pseudo-ground state transformation for forms $\calE_\Om^{V_c}$ performed in Lemma \ref{closability} together with an improved Sobolev inequality. This transformation has the considerable effect to mutate forms $\calE_\Om^{V_c}$ to Dirichlet forms and to mutate $e^{-tL_{V_c}^\Om}$ to Markovian ultracontractive semigroup on some weighted Lebesgue  space. The analysis of the transformed forms will then lead us to get satisfactory results concerning estimating their heat kernels (sharply) and hence to reveal  properties of minimal solutions.\\
As a first step  we proceed to prove that Sobolev inequality holds for the $w_c$-transform of the form $\calE_\Om^{V_c}$. As a byproduct we obtain that the semigroup of the transformed from is ultracontractive and then very interesting upper bound for the heat kernel are derived.
\begin{theo}
\begin{enumerate}
\item Let $0<c<c^*$ and $p=\frac{d}{d-\alpha}$. Then the following Sobolev inequality holds
\begin{eqnarray}
\parallel f^2\parallel_{ {L^p}(w_c^2dx)}\leq AQ_\Om^{c}[f],\ \forall\,f\in D(Q_\Om^c).
\label{w-sob}
\end{eqnarray}
\item For $c=c^*$ let  $1<p<\frac{d}{d-\alpha}$. Then the following Sobolev inequality holds
\begin{eqnarray}
\parallel f^2\parallel_{ {L^p}(w_{c^*}^2dx)}\leq AQ_\Om^{c^*}[f],\ \forall\,f\in D(Q_\Om^{c^*}).
\label{w-sob2}
\end{eqnarray}
\item For each $t>0$, and  $0<c\leq c^*$ the operator $T_{t,\Om}^{w_c}$ is ultracontractive.
\item For every $0<c<c^*$,  there is a finite constant $C>0$ such that
\begin{eqnarray}
0<p_t^{L_{V_c}^\Om}(x,y)\leq \frac{C}{t^{\frac{d}{\alpha}}} w_c(x)w_c(y),\ a.e.\ {on}\ \Om\times\Om,\ \forall\,t>0.
\label{UppBound}
\end{eqnarray}
\item For $c=c^*$, and $1<p<\frac{d}{d-\alpha}$ there is a finite constant $C>0$ such that
\begin{eqnarray}
0<p_t^{L_{V_{c^*}}^\Om}(x,y)\leq \frac{C}{t^{\frac{p}{p-1}}} w_{c^*}(x)w_{c^*}(y),\ a.e.\ {on}\ \Om\times\Om,\ \forall\,t>0.
\label{UppBound2}
\end{eqnarray}
\end{enumerate}
\label{UC}
\end{theo}
\begin{proof}
1) and 2): Let $0<c<c^*$. From Hardy's inequality we derive
\begin{eqnarray}
(1-\frac{c}{c^*})\calE_\Om[f]\leq\calE_\Om^{V_c}[f],\ \forall\,f\in\Wo.
\end{eqnarray}
Now we use the known fact that $\Wo$ embeds continuously into $L^{\frac{2d}{d-\alpha}}$, to obtain the following Sobolev's inequality
\begin{eqnarray}
(\int_\Om |f|^{\frac{2d}{d-\alpha}}\,dx\big)^{\frac{d-\alpha}{d}}\leq C\calE_\Om^{V_c}[f],\ \forall\,f\in\Wo.
\end{eqnarray}
An application of H\"older's inequality together with Lemma \ref{closability} and the fact that $\Om$ is bounded, yield then inequality (\ref{w-sob}).\\
Towards proving Sobolev's inequality in the critical case we use the improved Hardy--Sobolev inequality, due to Frank--Lieb--Seiringer [Theorem 2.3]: For every $1\leq p<\frac{d}{d-\alp}$ there is a constant $S_{d,\alp}(\Om)$ such that
\begin{eqnarray}
(\int |f|^{2p}\,dx)^{1/p}\leq S_{d,\alp}(\Om)\big(\calE_\Om[f]-c^*\int_\Om\frac{f^2(x)}{|x|^\alpha}\,dx\big),\ \forall\,f\in\Wo,
\label{ISI}
\end{eqnarray}
and the rest of the proof runs as before.\\
3) and 4): As $Q_\Om^c$ is a Dirichlet form, by the standard theory of Markovian semigroups, it is known (see \cite[p.75]{davies-book}) that Sobolev inequality implies ultracontractivity of $T_{t,\Om}^{w_c}$ together with the bound
\begin{eqnarray}
\|T_{t,\Om}^{w_c}\|_{L^2(\Om,w_c^2dx),L^\infty(\Om)}\leq\frac{c}{t^{d/\alpha}},\ t>0.
\end{eqnarray}
Now ultracontractivity in turns implies that the semigroup $e^{-tA_\Om^{w_c}}$ has a nonnegative symmetric (heat) kernel, which we denote by $q_t^\Om$ and  the latter estimate yields by \cite[p.59]{davies-book})
\begin{eqnarray}
0\leq q_t^\Om(x,y)\leq \frac{c}{t^{d/\alpha}},\ a.e.,\ \forall\,t>0.
\end{eqnarray}
Recalling formula (\ref{TransformedKernel}):
\begin{eqnarray}
q_t^\Om(x,y)=\frac{p_t^{L_{V_c}^\Om}(x,y)}{w_c(x)w_c(y)},\ a.e.,
\end{eqnarray}
yields the upper bounds (\ref{UppBound}) and (\ref{UppBound2}).\\
The proof of 5. is similar to the latter one so we omit it.
\end{proof}
At this stage we turn our attention to establish lower a bound for the heat kernels $p_t^{L_{V_c}^\Om}$.\\
Let us first observe that from the definition, the Dirichlet form $Q_\Om^c$ is nothing else but the part of the form $Q^c$ on $\Om$, i.e., $Q^c_\Om=Q^c|_{D(Q^c_\Om)}$ where
$$
\dom(Q^c_\Om)=\{f\in D(Q^c):f\equiv0\,\, {\rm on }\,\, \Omega^c\}.
$$
Since $Q^c$ is a Dirichlet form and $q_t$ is continuous there exists a Hunt process on $\mathbb{R}^d$ such that $$\mathbb{P}^x(X_t\in A)=\int_A q_t(x,y) w_c^2(y)\,dy, \quad A\in\mathcal{B}(\mathbb{R}^d).$$
By positivity of $w_c$ and the Dynkin-Hunt formula we get
$$q_t^\Om(x,y)=q_t(x,y)-\mathbb{E}^x[\tau_\Omega<t,q_{t-\tau_\Om}(X_{\tau_\Om},y)],\quad x,y\in \Om,$$
where $\tau_\Omega=\inf\{t>0:X_t\notin\Om\}.$

Let $S(t,x):=|x|^{\beta(c)}+t^{\beta(c)/\alpha}$ and $H(t,x):=1+w_c(xt^{-1/\alpha})=w_c(x)S(t,x)$.  We know from \cite[Lemma 5.1, Theorem 1.1]{BogdanHardy} together with formula (\ref{TransformedKernel}) that
\begin{equation}\label{eq:HKest_q}q_t(x,y)\approx S(t,x)\left(t^{-d/\alpha}\wedge\frac{t}{|x-y|^{d+\alpha}}\right)S(t,y),\quad t>0, \,x,y\in\mathbb{R}^d.
\end{equation}
\begin{theo}
For every $0<c\leq c^*$,  every compact subset $K\subset\Om$ and every $t>0$, there is a finite constant $\kappa_t=\kappa_t(K)>0$ such that
\begin{eqnarray}
p_t^{L_{V_c}^\Om}(x,y)\geq \kappa_tw_c(x)w_c(y),\ a.e.\ {on}\ K\times K,\ \forall\,t>0.
\label{LowBound}
\end{eqnarray}
\label{ThmLowBound}
\end{theo}
\begin{proof}
Since $0\in \Om$ we may and do assume that $0\in K$ (we can consider infimum of $p_t^{L_{V_c}^\Om}$ on the larger set).

First we prove the lower bound for $q_t^\Om$ on small balls around zero and small $t>0$. Let $0<r<1$ be such that $\overline{B_{4r}}\subset\Om$ and $x,y\in B_{r}$. Then Dynkin--Hunt formula leads to
\begin{eqnarray}
q_t^\Om(x,y)&=&q_t(x,y) - \mathbb{E}^x[t>\tau_{\Om},q_{t-\tau_{\Om}}(y,X_{\tau_{\Om}})]\nonumber\\
&\geq& q_t(x,y) - \sup_{s\leq t,z\in \Om^c}q_{s}(y,z).
\end{eqnarray}
Since $S(\cdot,y)$ is increasing and $|y-z|>|z|/2>r$ for $z\in\Om^c$ by \eqref{eq:HKest_q} we obtain
\begin{align*}
\sup_{s\leq t,z\in \Om^c}q_{s}(y,z)&\geq c_1 \sup_{z\in\Om^c}S(t,y)\frac{t}{|z|^{d+\alpha}}S(t,z)\\
&\geq c_1S(t,y) \frac{t}{r^{d+\alpha}}\left( 1 + t^{\beta(c)/\alpha}\right).
\end{align*}
Hence and again \eqref{eq:HKest_q} yields for $t\leq 1$
\begin{align*}
\frac{q_t^\Om(x,y)}{S(t,y)}&\geq c_2 t^{\beta(c)/\alpha}\left(t^{-d/\alpha}\wedge \frac{t}{|x-y|^{d+\alpha}}\right)-c_1 \frac{t}{r^{d+\alpha}}.
\end{align*}
For $|x-y|\leq t^{1/\alpha}$ and $t\leq T(r):=\left(\frac{c_2r^{d+\alpha}}{2c_1}\right)^{\alpha/(\alpha+d-\beta(c))}<r$ we get
\begin{align*}
\frac{q_t^\Om(x,y)}{S(t,y)}&\geq c_2 t^{(\beta(c)-d)/\alpha}-c_1 \frac{t}{r^{d+\alpha}}\geq \frac{c_2}{2} t^{(\beta(c)-d)/\alpha}.
\end{align*}
This implies
$$q_t^\Om(x,y)\geq c S(t,x)S(t,y)t^{-d/\alpha},\quad |x|,|y|\leq \frac{t^{1/\alpha}}{2}. $$ In consequence
$$p_t^{L_{V_c}^\Om}(x,y)\geq c H(t,x)H(t,y)t^{-d/\alpha}\geq c H(t,x)H(t,y)p_t^{L_{0}^\Om}(x,y),\quad |x|,|y|\leq \frac{t^{1/\alpha}}{2}.$$

Let $|x|\leq t^{1/\alpha}/2<|y|\leq r$. Set $D:=B_{t^{1/\alpha}/4}\setminus B_{t^{1/\alpha}/8}$. By Duhamel's formula  and estimates of $p_t^{L_0^\Om}$  $$p_t^{L_{V_c}^\Om}(z,y)\geq p_t^{L_0^\Om}(z,y)\geq c \frac{t}{|y|^{d+\alpha}}, \quad z\in D.$$ By the semigroup property
\begin{align*}
p_t^{L_{V_c}^\Om}(x,y)&\geq \int_{D}p_{t/2}^{L_{V_c}^\Om}(x,z)p_{t/2}^{L_{V_c}^\Om}(z,y)\,dz\geq cH(t,x)t^{-d/\alpha}\frac{t}{|y|^{d+\alpha}}|D|\\&\geq c H(t,x)H(t,y)p_t^{L_{0}^\Om}(x,y).
\end{align*}
For $t^{1/\alpha}/2<|y|,|x|$. We get
$$p_t^{L_{V_c}^\Om}(x,y)\geq p_t^{L_0^\Om}(x,y)\geq \inf_{x,y\in K}p_t^{L_0^\Om}(x,y)=c_t(K)>0.$$
For $|x|\leq t^{1/\alpha}/2\leq r <|y|$ one can obtain by the semigroup property $p_t^{L_{V_c}^\Om}(x,y)\geq cH(t,x)c_t(K)$.
In particular we have
\begin{equation*}
p_t^{L_{V_c}^\Om}(x,y)\geq H(t,x)H(t,y)c_t(K),\quad t\leq T(r),\,x,y\in K.
\end{equation*}
If $t>T(r)$ we use the semigroup property to obtain
\begin{align*}
p_t^{L_{V_c}^\Om}(x,y)&\geq \int\int_{|z|,|w|\leq r} p_{T(r)/4}^{L_{V_c}^\Om}(x,z)p_{t-T(r/2)}^{L_{V_c}^\Om}(z,w)p_{T(r)/4}^{L_{V_c}^\Om}(w,y)dzdw\\ &\geq c(r,K)H(t,x)H(t,y)\inf_{|z|,|w|<r}p_{t-T(r/2)}^{L_{0}^\Om}(z,w),
\end{align*}
which ends the proof.
\end{proof}
We are now in position to describe the exact behavior, in space variable, of the minimal solution  of equation (\ref{heat1}), especially near $0$.
\begin{theo}
\begin{enumerate}
\item For every $t>0$ there is a finite constant $c_t>0$ such that,
\begin{eqnarray}
u(t,x)\leq c_tw_c(x),\ a.e.\ on\ \Om.
\label{UB}
\end{eqnarray}
It follows in particular that $u(t,x)$ is bounded away from zero.
\item For every $t>0$, there are finite constants $c_t,\ c'_t>0$ such that
\begin{eqnarray}
c'_tw_c(x)\leq u(t,x)\leq c_tw_c(x),\ a.e.\ {\rm near}\ 0.
\label{Sharpo}
\end{eqnarray}
\end{enumerate}
Hence $u(t)$  has a standing singularity at $0$.
\label{SharpLoc}
\end{theo}
\begin{proof}
The upper bound (\ref{UB}) follows from Theorem \ref{UC}-4). Let us now  prove the lower bound.\\
Let $K$ be a compact subset of $\Om$ containing $0$ such that Lebesgue measure of the set $\{x\in K\colon\,u_0(x)>0\}$ is nonnegative.\\
Let $\kappa_t$ be as in (\ref{LowBound}), then
\begin{align*}
u(t,x)&=\int_\Om p_t^{L_{V_c}^\Om}(x,y)u_0(y)\,dy\geq\int_K p_t^{L_{V_c}^\Om}(x,y)u_0(y)\,dy\geq
\kappa_t w_c(x)\int_K w_c(y)u_0(y)\,dy\nonumber\\
&\geq c_t'w_c(x),\ a.e.\ \text{ on }\ K,
\end{align*}
with $c_t'>0$, which was to be proved.
\end{proof}
The local sharp estimate (\ref{Sharpo}) leads  to a sharp global regularity property of the minimal solution, expressing thereby the smoothing effect of the semigroup $e^{-tL_{V_c}^\Om}$.
\begin{prop}
\begin{enumerate}
\item For every $t>0$, the minimal solution $u(t)$ lies in the space $L^p(\Om),\ p\geq 1$  if and only if $1\leq p< \frac{d}{\beta}$.
\item The operator $e^{-tL_{V_c}^\Om},\ t>0$ maps continuously $L^2(\Om)$ into $L^p(\Om)$ for every $2\leq p< \frac{d}{\beta}$.
\item The operator  $e^{-tL_{V_c}^\Om}: L^q(\Om)\to L^p(\Om),\ t>0$ is bounded for every $\frac{d}{d-\beta}<q<p<\frac{d}{\beta}$.
\item The operator $L_{V_c}^\Om$ has compact resolvent. Set $(\varphi_k^{L_{V_c}})_k$ its eigenfunctions. Then   $(\varphi_k^{L_{V_c}})_k\subset L^p(\Om)$ for every $p< \frac{d}{\beta}$.
\end{enumerate}
\end{prop}
\begin{proof}
The first assertion is an immediate consequence of Theorem (\ref{SharpLoc}).\\
2): Let $u_0\in L^2(\Om)$ and $p$ as described in the assertion. Thanks to the upper bounds (\ref{UppBound})-(\ref{UppBound2}) a straightforward computation leads to
\begin{eqnarray}
\int_\Om \left|e^{-tL_{V_c}^\Om} u_0(x)\right|^p\,dx\leq c_t(\int_\Om w_c|u_0|\,dx)^p\int_\Om w_c^p\,dx\leq C(\int_\Om u_0^2\,dx)^{p/2}.
\end{eqnarray}
3): Follows from  Riesz-Thorin interpolation theorem.\\
4): We claim that for each $t>0$ the operator $e^{-tL_{V_c}^\Om}$ is a Hilbert--Schmidt. Indeed, the upper bound (\ref{UppBound}) lead to
$$
\int_\Om\int_\Om \left(p_t^{L_{V_c}^\Om}\right)^2(x,y)\,dx\,dy\leq C\int_\Om w_c^2(x)\,dx
\cdot\int_\Om w_c^2(x)\,dx<\infty,
$$
and the claim is proved. Hence $L_{V_c}^\Om$ has compact resolvent. The claim about eigenfunctions follows from assertion 2.
\end{proof}
\begin{rk}
{\rm We claim that $\dot\calE_\Om^{V_{c^*}}$ is not closed. Indeed, utilizing the inequality $ p_t^{L_{V_c}^\Om}\geq p_t^{L_0^\Om}$ we conclude that the semigroup $e^{-tL_{V_c}^\Om}$ is irreducible for each $t>0$. Consequently, the smallest eigenvalue of $L_{V_c}^\Om$ is non-degenerate, i.e. its eigenspace has dimension one and is generated by a nonnegative function, say $ \varphi^{L_{V_{c^*}}^\Om}$. If $\dot\calE_\Om^{V_{c^*}}$ were closed, then the ground state  $ \varphi^{L_{V_{c^*}}^\Om}$ would be in the space $\Wo$ and hence by the improved Sobolev inequality we would get $\int_\Om (\varphi^{L_{V_{c^*}}^\Om})^2(x)\,dx<\infty$. However, from the lower bound (\ref{LowBound}), we obtain for each small ball around zero
$$
\int_\Om \left(\varphi^{L_{V_{c^*}}^\Om}\right)^2(x)\,dx\geq C\int_B w_{c^*}^2(x)V_{c^*}(x)\,dx=\infty,
$$
leading to a contradiction.

}
\label{NotClosed}
\end{rk}
The already established upper estimate for the heat kernel enables one to extend the semigroup to a larger class of initial data.
\begin{theo}
\begin{enumerate}
\item The semigroup $e^{-tL_{V_c}^\Om},\ t>0$ extends to a bounded linear  operator from $L^1(\Om,w_cdx)$ into $L^2(\Om)$.
\item  The semigroup $e^{-tL_{V_c}^\Om},\ t>0$ extends to a bounded linear operator from $L^p(\Om,w_cdx)$ into $L^p(\Om)$ for every $1\leq p<\infty$.
\item  The semigroup $e^{-tL_{V_c}^\Om},\ t>0$ extends to a bounded linear  semigroup from $L^p(\Om,w_cdx)$ into $L^p(\Om,w_cdx)$ for every $1\leq p<d/3$.
\end{enumerate}
\end{theo}
\begin{proof}
Having estimate (\ref{UppBound}) in hands, a straightforward computation yields
\begin{eqnarray*}
\int_\Om (e^{-tL_{V_c}^\Om}u_0)^2\,dx\leq c_t\int_\Om w_c^2\,dx\cdot\big(\int_\Om |u_0|w_c\,dy\big)^2,\ \forall\,t>0,
\end{eqnarray*}
and assertion 1. is proved.\\
Similarly, using H\"older's inequality we achieve
\begin{align*}
|e^{-tL_{V_c}^\Om}u_0(x)|^p&\leq \int_\Om p_t^{L_{V_c}^\Om}(x,y)\,dy\int_\Om p_t^{L_{V_c}^\Om}(x,y) |u_0|^p\,dy
\leq c_tw_c^2(x)\int_\Om w_c(y)\,dy\int_\Om |u_0|^pw_c\,dy\\
&\leq c_t \left(\int_\Om |u_0|^pw_c\,dy\right) w_c^2(x).
\end{align*}
Integrating w.r.t. $x$, we obtain assertion 2.\\
Assertion 3. can be proved in the same way.
\end{proof}

\section{Blow-up of nonnegative solutions on open sets in the supercritical case}
In this section we shall make use of the lower bound for the heat kernel as well as for nonnegative solutions in the critical case on bounded open sets, which we established in the last section, to show that for $c>c^*$ any nonnegative solution of the heat equation (\ref{heat1}) on arbitrary open sets containing zero  blows up completely and instantaneously. This result accomplishes the corresponding one for bounded sets with Lipschitz boundary \cite{benamor-kenzizi} and the case where $\Om=\R^d$ \cite{BogdanHardy}, so that to get a full picture concerning existence and nonexistence of nonnegative solutions for Dirichlet fractional Laplacian with Hardy potentials.\\
However, the idea of the proof deviates  from those developed in \cite{benamor-kenzizi,BogdanHardy}. Our proof relies on the sofar established lower bounds for $p_t^{L_{V_{c^*}}^\Om}$ and for nonnegative solutions on balls.\\
Henceforth we fix a nonempty open set $\Om\subset\R^d$ containing zero and $c>0$.\\
Let $V\in L_{loc}^1(\Om,dx)$ be a nonnegative potential. We set  $W_k:=V\wedge k$ and  $(P_k)$ the heat equation corresponding to the Dirichlet fractional Laplacian perturbed by $-W_k$ instead of  $-V$:
\begin{eqnarray}
\label{heat-app}
(P_k)\colon\left\{\begin{gathered}
-\frac{\partial u}{\partial t}=L_0^\Om u - W_k u,\quad \hbox{in } (0,T)\times\Om,\\
u(t,\cdot)=0,\ on~~~\Omega^c,\ \forall\,0<t<T\leq\infty\\
u(0,x)= u_{0}(x),~~~{\rm for}\ a.e.\ x\in \R^d,
\end{gathered}
\right.
\end{eqnarray}
Denote by  $L_k$  the selfadjoint operator associated to the closed quadratic form  $\calE_\Om-W_k$  and $u_k(t):=e^{-tL_k}u_0,\ t\geq 0$ the nonnegative semigroup solution of problem $(P_k)$. Then
$u_k$ satisfies Duhamel's formula:
\begin{eqnarray}
u_k(t,x)&=&e^{-tL_0^\Om}u_0(x)+\int_0^t\int_\Om p_{t-s}^{L_0^\Om}(x,y)u_k(s,x)V_k(y)\,dy\,ds,\ \forall\,t>0,
\label{duhamel}
\end{eqnarray}
Let us list the properties of the sequence $(u_k)$ and establish existence of the minimal solution.
\begin{lem}
\begin{itemize}
\item[i)] The sequence $(u_k)$ is increasing.
\item[ii)] If $u$ is any nonnegative solution of problem (\ref{heat2}) then $u_k\leq u,\ \forall\,k$. Moreover $u_\infty:=\lim_{k\to\infty}u_k$ is a nonnegative solution of problem (\ref{heat2}) as well.
\end{itemize}
\label{domination}
\end{lem}
The proof runs as the one corresponding to the case of bounded domains (see \cite{benamor-kenzizi}), so we omit it.\\
We recall that we use the notation $u(t)$ to designate the minimal solution $u_\infty(t)$.
\begin{rk}
{\rm Let $0<c\leq c^*$. Owing to the lower bound (\ref{LowBound}) together with the fact that $p_t^{L_{V_c}^\Om}\geq p_t^{L_{V_c}^B}$ for any ball such that $B\subset\Om$ we automatically get: for every compact subset $K\subset\Om$ and every $t>0$, there is a finite constant $\kappa_t=\kappa_t(K)>0$ such that
\begin{eqnarray}
p_t^{L_{V_c}^\Om}(x,y)\geq \kappa_tw_c(x)w_c(y),\ a.e.\ {on}\ K\times K,\ \forall\,t>0.
\label{LowBoundGen}
\end{eqnarray}
Hence
\begin{eqnarray}
u(t)\geq c_t w_c(x)\ a.e.\ \text{near}\ 0.
\label{LowMinGen}
\end{eqnarray}
Thus, for any open nonempty subset the minimal solution has a standing singularity at $0$.
}
\end{rk}
Let us establish a Duhamel formula for the minimal solution.
\begin{lem}
Let $u$ be the minimal solution of equation (\ref{heat1}) with $c\geq c^*$. Then $u$ satisfies the following Duhamel's formula:
\begin{eqnarray}
u(t,x) = e^{-tL_{V_{c^*}}^\Om}u_0(x) +(c-c^*)\int_0^t\int_{\Om} p_{t-s}^{L_{V_{c^*}}^\Om}(x,y)u(s,y)|y|^{-\alpha}\,ds\,dy,\ \forall\,t>0,\ a.e.\,x.
\label{Rep}
\end{eqnarray}
\end{lem}
\begin{proof}
Set $W_k^*=V_{c^*}\wedge k$. Then
\begin{eqnarray}
u_k(t,x)=e^{-tL_k^\Om}u_0(x)= e^{-tL_{W_k^*}^\Om}u_0(x) +
 \int_0^t\int_\Om p_{t-s}^{L_{W_k^*}^\Om}(x,y)u_k(s)(W_k-W_{k}^*)\,dy\,ds.
\end{eqnarray}
A simple calculation shows that the sequence $(W_k-W_{k}^*)$ is increasing. As the minimal solution is the limit of the $u_k$'s, the result follows by application of monotone convergence theorem.
\end{proof}
We have sofar collected enough material to announce the main theorem of this section.
\begin{theo}
Assume that $c>c*$. Then the heat equation (\ref{heat1}) has no nonnegative solutions. It follows, that the minimal solution blows up completely and instantaneously.
\end{theo}
\begin{proof}
Assume that a nonnegative solution $u$ exists. Relying on Lemma \ref{domination}, we may and shall suppose that $u=u_\infty$. Put $c'=c-c^*>0$ and let $B$ be an open ball centered at $0$ such that $B\subset\Om$ and  $u_0\not\equiv 0$ on $B$.\\
Owing to the fact that $p_t^{L_{V_{c^*}}^\Om}\geq p_t^{L_{V_{c^*}}^B}$, the identity (\ref{Rep}) together with the lower bound from (\ref{LowMinGen})  we obtain
\begin{eqnarray}
u(t,x)\geq e^{-tL_{V_{c^*}}^B}u_0(x)\geq c_tw_{c^*}(x),\ a.e.\ {\rm on}\ B':=\frac{1}{2}B.
\end{eqnarray}
Using formulae (\ref{Rep}) and (\ref{LowBound}), once again together with the latter lower bound  we obtain

\begin{align}
u(t,x) &\geq c' \int_0^t c_s\int_{B} p_{t-s}^{L_{V_{c^*}}^B}(x,y)u(s,y)|y|^{-\alpha}\,ds\,dy \geq c'\int _0^t c_s\int_{B'} p_{t-s}^{L_{V_{c^*}}^B}(x,y)w_{c^*}(y)|y|^{-\alpha}\,ds\,dy\nonumber\\
&\geq c'w_{c^*}(x)\int_0^t c'_s\int_{B'} w_{c^*}^2(y)|y|^{-\alpha}\,ds\,dy.
\end{align}
However,
\begin{eqnarray}
\int_{B'} w_{c^*}^2(y)|y|^{-\alpha}\,dy=\infty,
\end{eqnarray}
and the solution blows up, which finishes the proof.
\end{proof}
{\bf Acknowledgement.} The author is very grateful to Tomasz Grzywny. He did the major part of the proof of Theorem \ref{ThmLowBound}.

\bibliography{biblio-HeatFDL}

\end{document}